\documentclass[a4paper]{article}
\usepackage{amsfonts}
\usepackage{mathrsfs}
\usepackage[english]{babel}
\usepackage[utf8x]{inputenc}
\usepackage[T1]{fontenc}
\usepackage{amsthm} 
\usepackage[a4paper,top=3cm,bottom=2cm,left=3cm,right=3cm,marginparwidth=1.75cm]{geometry}

\usepackage{enumitem}
\usepackage{mathtools}
\usepackage{amsmath}
\usepackage{graphicx}
\usepackage[colorinlistoftodos]{todonotes}
\usepackage[colorlinks=true, allcolors=blue]{hyperref}

\newtheorem{theorem}{Theorem}[section]
\newtheorem*{theorem-no-number}{Theorem}

\newtheorem*{conj}{Conjecture}
\newtheorem{lemma}{Lemma}[section]
\newtheorem{corollary}{Corollary}[section]
\newtheorem{definition}{Definition}[section]
\newtheorem*{definition-no-number}{Definition}

\title{On $1$-bridge braids, satellite knots, the manifold $v2503$ and non-left-orderable surgeries and fillings}
\author{Zipei Nie}

\begin{document}
\maketitle

\begin{abstract}
We define the property (D) for nontrivial knots. We show that the fundamental group of the manifold obtained by Dehn surgery on a knot $K$ with property (D) with slope $\frac{p}{q}\ge 2g(K)-1$ is not left orderable. By making full use of the fixed point method, we prove that (1) nontrivial knots which are closures of positive $1$-bridge braids have property (D); (2) L-space satellite knots, with positive $1$-bridge braid patterns, and companion with property (D), have property (D); (3) the fundamental group of the manifold obtained by Dehn filling on $v2503$ is not left orderable. Additionally, we prove that L-space twisted torus knots of form $T_{p,kp\pm 1}^{l,m}$ are closures of positive $1$-bridge braids.
\end{abstract} 
\section{Introduction}\label{introduction}

A nontrivial group $G$ is called left orderable if it admits a total ordering $\le$ which is invariant under left multiplication, which is, $g\le h$ implies $fg\le fh$ for any $f,g,h$ in $G$. The trivial group is not left orderable by convention. A rational homology $3$-sphere $Y$ is called an L-space if the hat version of its Heegaard Floer homology has rank $|H_1(Y)|$. Boyer, Gordon and Watson conjectured \cite{BGW} that an irreducible rational homology $3$-sphere is an L-space if and only if its fundamental group is not left orderable. This conjecture is verified \cite{BGW} for geometric, non-hyperbolic $3$-manifolds and $2$-fold branched covers of non-split alternating links.

One way to construct hyperbolic L-spaces is via Dehn surgery. A knot $K$ in $S^3$ is called a (positive) L-space knot if there exists a positive Dehn surgery along $K$ which yields an L-space. For any nontrivial L-space knot $K$ in $S^3$, the Dehn surgery along $K$ with slope $\frac{p}{q}$ yields an L-space if and only if $\frac{p}{q}\ge 2g(K)-1$ \cite{He,Oz}.

A twisted torus knot $T_{p,q}^{l,m}$ is defined as the knot obtained from the $(p,q)$-torus knot by twisting $l$ strands $m$ full times. It can be realized as the closure of the braid
$$(\sigma_{l-1}\sigma_{l-2}\cdots \sigma_2\sigma_1)^{l m}(\sigma_{p-1}\sigma_{p-2}\cdots\sigma_2 \sigma_1)^{q}$$
on $q$ strands. For positive integers $p,k,l,m$ with $0<l<p$, Vafaee proved that \cite{Va} the twisted torus knot $T_{p,kp\pm 1}^{l,m}$ is an L-space knot if and only if 
\begin{enumerate}[label=(\alph*)]
\item $l=p-1$, or
\item $m=1$ and $l=2$, or
\item $m=1$ and $l=p-2$.
\end{enumerate}
A number of research papers \cite{Chr,Clay,  Ichi1, Ichi2,Jun, Na, Nie, Tran1} have been devoted to study non-left-orderable surgeries on subfamilies of the L-space twisted torus knots $T_{p,kp\pm 1}^{l,m}$. And recently, Tran proved \cite{Tran2} that the fundamental group of the manifold obtained by Dehn surgery on an L-space twisted torus knot of form $T_{p,kp\pm 1}^{l,m}$ with slope at least $2 g(T_{p,kp\pm 1}^{l,m})-1$ is not left orderable.

A knot in $S^1 \times D^2$ is called a (positive) $1$-bridge braid, if it is a positive braid, and the knot is isotopic to a union of two arcs $\gamma\cup \delta$, where $\gamma\subset \partial(S^1\times D^2)$ transverses to each meridian, and the bridge $\delta$ is properly embedded in some meridional disk $D$. A $1$-bridge braid $B(\omega,t,b)$ can be represented \cite{Ga} by the braid word $$(\sigma_b\sigma_{b-1} \cdots \sigma_2\sigma_1)(\sigma_{\omega-1} \sigma_{\omega-2}\cdots \sigma_2\sigma_1)^{t}$$ on $\omega$ strands, where $0\le b\le \omega-2$ and $t\ge 1$. It was shown \cite{GLV} that nontrivial knots which are closures of $1$-bridge braids, are L-space knots. Liang proved \cite{Liang} that if 
\begin{enumerate}[label=(\alph*)]
\item $t=1+m\omega$ and $b=2k$, or
\item $\omega=2n+1$, $t=2n-1+m\omega$ and $b=2k$, or
\item $\omega=2n$, $t=2n-2+m\omega$ and $b=2k-1$,
\end{enumerate}
then fundamental group of the manifold obtained by Dehn surgery on the closure of $B(\omega,t,b)$ with slope at least $\omega t+b-t$ is not left orderable.

In Section \ref{braid-word-section}, we will prove the following theorem algebraically using the Markov braid theorem.

\begin{theorem}\label{main-braid-word}
If one of the following conditions holds for the positive integers $p,q,l,n$,
\begin{enumerate}[label=(\alph*)]
    \item $q\ge l+1= p$,
    \item $q=l$ divides $n$, and $\gcd(p,q)=1$, 
    \item $l=n=2$ and $\gcd(p,q)=1$,
    \item $l=n=p-2$ and $q= kp\pm 1$ for some positive integer $k$,
\end{enumerate}
then there exist integers $\omega,t,b$ with $0\le b \le \omega-2$ and $t\ge 1$, such that the closure of the braid 
$$(\sigma_{l-1}\sigma_{l-2}\cdots\sigma_2\sigma_1)^n (\sigma_{p-1}\sigma_{p-2}\cdots\sigma_2\sigma_1)^q$$
and the closure of the braid
$$(\sigma_b\sigma_{b-1}\cdots\sigma_2\sigma_1)(\sigma_{\omega-1}\sigma_{\omega-2}\cdots\sigma_2\sigma_1)^t$$
represent the same oriented link in $S^3$.
\end{theorem}

In particular, the following corollary follows directly from Theorem \ref{main-braid-word} and Vafaee's result.

\begin{corollary}\label{braid-word-corollary}
L-space twisted torus knots of form $T_{p,kp\pm 1}^{l,m}$ for positive integers $p,k,l,m$ with $0<l<p$ are closures of $1$-bridge braids. 
\end{corollary}

In Section \ref{main-section}, we will prove the following theorem.

\begin{theorem}\label{main}
The fundamental group of the manifold obtained by Dehn surgery on the nontrivial knot $K$ which is the closure of a $1$-bridge braid $B(\omega, t,b)$ with slope $\frac{p}{q}\ge \omega t+b-t-\omega$ is not left orderable.
\end{theorem}

As the surface obtained by applying Seifert's algorithm to a positive diagram is a minimal genus Seifert surface \cite{Cro}, the genus of $K$ equals $\frac{1}{2}(\mbox{\#crossings}-\mbox{\#strands}+1)=\frac{1}{2}((\omega-1)t+b-\omega+1)$. So we have $\omega t+b-t-\omega= 2g(K)-1$. Therefore, Theorem \ref{main} generalizes Tran's result on L-space twisted torus knots of form $T_{p,kp\pm 1}^{l,m}$ and Liang's result on closures of certain $1$-bridge braids.

Let us list some approaches of proving the non-left-orderability in the literature.

\begin{enumerate}
    \item Brute force search in possible left orders --- Whenever we find a selected element for which we cannot decide its positivity, we consider the problem in two scenarios, in one we assume the element is positive, in the other we assume it is negative. See \cite{Ba} as an example of extensive use of brute force search. For any finitely generated non left orderable group, the brute force search always terminates and produce a proof of non-left-orderability \cite{Con}.
    \item The fixed point method --- It is known that \cite{Hol} a countable group is left orderable if and only if it acts on the real line by orientation preserving homeomorphisms without global fixed points. For a certain element $k$ in the group, if there exists a real number $s$ fixed by $k$, then we try to prove that $s$ is a global fixed point; otherwise, all conjugates of $k$ have the same sign. See \cite{Ichi1, Tran2} for example.
    \item \cite[Theorem~2.1]{Clay} gives a criteria to deduce the non left orderability of $\pi_1(M(\frac{p}{q}))$ from the signs of $\mu^{p_0} \lambda^{q_0}$ and $\mu^{p_1}\lambda^{q_1}$ where $\frac{p_0}{q_0}<\frac{p}{q}<\frac{p_1}{q_1}$ under certain conditions. This theorem is used to derive the non left orderability of a two-parameter family of groups.
    \item For a Seifert fibred space $M$, one can construct \cite{BRW} a horizontal foliation on $M$ from the representation $\pi_1(M)\to \mbox{Homeo}^+ (\mathbf{R})$ under certain conditions. Since it is known \cite{EHN, JN, Na} which Seifert bundles admit horizontal foliations, one can derive the non-left-orderability of $\pi_1(M)$.
\end{enumerate}

To prove Theorem \ref{main}, we make full use of the second technique as follows. Assume the theorem does not hold. First, we find an element $k$ in the group which represents a self-homeomorphism on $\mathbf{R}$ without fixed points. By continuity, the self-homeomorphism is pointwise greater than or pointwise less than the identity map. Naturally, we define the pointwise partial order $\le_C$ on the set of self-homeomorphisms, or equivalently, on the group with a left order, then the partial order $\le_C$ is invariant not only under left multiplication but also under right multiplication.

Suppose $k$ represents a self-homeomorphism which is pointwise greater than (resp., pointwise less than) the identity map, by $k$ we generate a subset of the left orderable group, which is closed under multiplication, conjugation and taking roots, and whose elements have the same property. Then, this subset does not contain the identity element. Equivalently, we take this subset as a positive cone to define a preorder $\le_k$ on the group. As $\le_C$ is an extension of $\le_k$ (resp., the inverse of $\le_k$), the preorder $\le_k$ has to be a partial order. However, we prove the negation.

Notice that the preorder $\le_k$ does not depend on the choice of left order, but only on the group element $k$. So our method is computationally lighter and more human-readable than the first approach. 

Knot cabling is one way to produce L-space knots. By definition, the $(p,q)$-cable of a knot $K$ is the satellite knot $P(K)$ with pattern $P$ being the $(p,q)$-torus knot. The $(p,q)$-cable of a nontrivial knot $K$ is an L-space knot if and only if \cite{He,Hom} $K$ is an L-space knot and $\frac{q}{p}\ge 2 g(K)-1$. 

Hom, Lidman and Vafaee extended \cite{Hom2} the above result to the satellite knot $P(K)$ with pattern $P$ being $1$-bridge braid. They proved that, if the pattern $P$ is the $1$-bridge braid $B(\omega, t,b)$ and the companion $K$ is nontrivial, then the satellite knot $P(K)$ is an L-space knot if and only if $K$ is an L-space knot and $\frac{t}{\omega} \ge 2g(K)-1$. (The formula in the original paper is $\frac{b+t\omega}{\omega^2}\ge 2g(K)-1$. As $0\le b\le \omega-2$ and $\omega,t, b$ and $g(K)$ are integers, it is equivalent to the formula here.) They also proved that if the pattern $P$ is a negative braid, then the satellite knot $P(K)$ is not a (positive) L-space knot.

Clay and Watson introduced \cite{Clay2} the concept of $r$-decayed knots. Then they proved that the $(p,q)$-cable of an $r$-decayed knot is $pq$-decayed if $\frac{q}{p}>r$, and that sufficiently positive surgeries on decayed knots yield manifolds with non-left-orderable fundamental groups. The decayed knots were defined as follows to relate the left orders on the knot group and the left orders on the fundamental group of the manifold obtained by Dehn surgery.

\begin{definition-no-number}
For a nontrivial knot $K$ with $\mu$ and $\lambda$ representing a meridian and a longitude in the knot group, we say $K$ is $r$-decayed, if for any positive cone $P$ in the knot group, either $P\cap S_r=S_r$ or $P\cap S_r=\emptyset$, where $S_r$ is the set $\left\{\mu^p\lambda^q| \frac{p}{q}\ge r\right\}$.
\end{definition-no-number}

As an analog of the concept of decayed knots, we define the property (D) for nontrivial knots as follows. 

\begin{definition-no-number}
For a nontrivial knot $K$ with $\mu$ and $\lambda$ representing a meridian and a longitude in the knot group, we say $K$ has property (D) if 
\begin{enumerate}[label=(\alph*)]
    \item for any homomorphism $\rho$ from $\pi_1(S^3\setminus K)$ to $\mbox{Homeo}_+(\mathbf{R})$, if $s\in\mathbf{R}$ is a common fixed point of $\rho(\mu)$ and $\rho(\lambda)$, then $s$ is a fixed point of every element in $\pi_1(S^3\setminus K)$, and,
    \item $\mu$ is in the root-closed, conjugacy-closed submonoid generated by $\mu^{2g(K)-1}\lambda$ and $\mu^{-1}$.
\end{enumerate}
\end{definition-no-number}

Two major differences between our definition of property (D) and the definition of decayed knots are the follows.
\begin{enumerate}
    \item In the viewpoint of fixed points, there is a common condition on left orders on the fundamental groups of the manifolds obtained by all nonzero surgeries, but not on left orders on the knot group: $\rho(\mu)$ and $\rho(\lambda)$ have the same set of fixed points. Without referring to the homomorphism $\rho$, the condition can be rewritten as, for any element $g$, we have $\mu g\le g$ if and only if $g\le\lambda g$ if the surgery slope is positive, and $\mu g\le g$ if and only if $\lambda g\le g$ if the surgery slope is negative.
    
    \item We impose our condition on a preorder instead of the set of left orders. One technical difficulty of using the left orders is that they are total orders, so the relation between left orders on the knot group and left orders on the fundamental group of the manifold obtained by Dehn surgery is complicated.
\end{enumerate}

If a knot $K$ has property $(D)$, we will prove that the fundamental group of the manifold obtained by Dehn surgery on $K$ with slope $\frac{p}{q}\ge 2g(K)-1$ is not left orderable.

The following theorem is slightly stronger than Theorem \ref{main}. We only need to work in the knot group, instead of the fundamental group of the manifold obtained by Dehn surgery.

\begin{theorem}\label{main-variant}
Nontrivial knots which are closures of $1$-bridge braids have property (D).
\end{theorem}

The purpose of defining decayed knots was to prove the non-left-orderability for L-space cable knots. Likewise, the purpose of defining property (D) is to prove the non-left-orderability for L-space satellite knots with $1$-bridge braid patterns. 

In Section \ref{satellite-section}, we will prove the following theorem by utilizing an element without fixed points.

\begin{theorem}\label{main-satellite}
Let the pattern $P$ be a $1$-bridge braid $B(\omega,t,b)$, and the companion $K$ be a nontrivial knot with property (D). Suppose that $\frac{t}{\omega} \ge 2g(K)-1$, then the satellite knot $P(K)$ has property (D).
\end{theorem}

Combining Theorem \ref{main-variant} and Theorem \ref{main-satellite}, we can generate many knots with property (D), which we may call them L-space iterated $1$-bridge braid knots. In particular, the L-space iterated torus knots have property (D). It is known that \cite{Gor} surgeries on iterated torus knots always yield graph manifolds, and that \cite{BC,Hanselman} a graph manifold has a non-left-orderable fundamental group if and only if it is an L-space. Historically, the proof of the if part used the fourth method listed above, and the third approach was also exercised \cite{Li} on certain pairs of $(p, q)$-cable knots of torus knots.

Culler and Dunfield \cite{Culler} found that every non-longitudinal Dehn filling on the hyperbolic $\mathbf{Q}$-homology solid torus $v2503$ is an L-space, and suggested the readers to prove that every Dehn filling on $v2503$ yields a manifold with a non-left-orderable fundamental group. Varvarezos \cite{Var} proved that, for a certain homological framing, the fundamental group of the manifold obtained by Dehn filling on $v2503$ with slope $\frac{p}{q}< -1$ is not left orderable. 

In Section \ref{v2503-section}, we will extend the result to every Dehn filling on $v2503$ also by utilizing an element without fixed points.

\begin{theorem}\label{main-v2503}
The fundamental group of an manifold obtained by Dehn filling on $v2503$ is not left orderable.
\end{theorem}

Motivated by these results, we speculate that the existence of a specific element without fixed points could be a key step to prove the non-left-orderability for fundamental groups of L-spaces obtained by Dehn surgeries. The author wonders whether the following statement holds.

\begin{conj}
Let $K$ be an L-space knot. Let $\mu$ and $\lambda$ represent a meridian and a longitude in the knot group. Then there exists an element $k$ in the knot group, such that for any homomorphism $\rho$ from $\pi_1(S^3\setminus K)$ to $\mbox{Homeo}_+(\mathbf{R})$ without global fixed points, if $\rho(\mu)$ and $\rho(\lambda)$ have the same set of fixed points, then $\rho(k)$ does not have fixed points.

\end{conj}

\section*{Acknowledgement}
The author thanks his advisor Professor Zoltán Szabó and colleague Konstantinos Varvarezos for helpful discussions.

\section{Knot group of $1$-bridge braid}\label{knot_group_sectiion}
In this section, we will investigate the knot group of a $1$-bridge braid. 

Let $B(\omega,t,b)$ be a $1$-bridge braid in $S^1 \times D^2$. The meridional disk $D$ of the solid torus $S^1 \times D^2$, the arc $\gamma$ and the bridge $\delta$ are explained in the definition of $1$-bridge braid in Section \ref{introduction}. Let the solid torus $S^1 \times D^2$ lie in the interior of a slightly larger solid torus $N(S^1 \times D^2)$. Then the knot group of $B(\omega, t,b)$ is $\pi_1(N(S^1 \times D^2)\setminus (\gamma \cup \delta))$. 

The bridge $\delta$ divides $D$ into two oriented embedded disks $D_x$ and $D_y$, where $\gamma$ intersects $\partial D_x$ positively in $b$ points, and $\gamma$ intersects $\partial D_y$ positively in $\omega-b-1$ points. The fundamental group of $(S^1\times D^2)\setminus \delta$ is the free group generated by $x$ and $y$, where $x$ represent a loop which intersects $D_x$ positively once and does not intersect $D_y$, and $y$ represent a loop which does not intersect $D_x$ and intersects $D_y$ positively once.

The space $N(S^1\times D^2) \setminus (S^1\times D^2)$ is homotopy equivalent to a torus, so its fundamental group is $\mathbf{Z}\times \mathbf{Z}$, generated by a meridian $r_\mu$ and a longitude $r_\lambda$ of $N(S^1\times D^2)$ in the knot group. 

The space $(S^1\times D^2)\setminus \delta$ is a genus two handlebody, so its fundamental group is freely generated by two elements $x$ and $y$, where $x$ represents a curve in $(S^1\times D^2)\setminus \delta$ which intersects $D_x$ positively once and does not intersect $D_y$, and $y$ represents a curve in $(S^1\times D^2)\setminus \delta$ which intersects $D_y$ positively once and does not intersect $D_x$.

The space $\partial(S^1\times D^2)\setminus\gamma$ is homotopy equivalent to a punctured torus, so its fundamental group is freely generated by two elements, which represent a homotopical meridian and a homotopical longitude of $S^1\times D^2$ which does not intersect the arc $\gamma$. In particular, the inclusion map $\partial(S^1 \times D^2)\setminus \gamma \hookrightarrow N(S^1\times D^2) \setminus (S^1\times D^2)$ induces an epimorphism in fundamental groups. 

The knot complement $N(S^1 \times D^2)\setminus (\gamma \cup \delta)$ can be obtained by gluing $(S^1\times D^2)\setminus \delta$ with $N(S^1\times D^2) \setminus (S^1\times D^2)$ along $\partial(S^1\times D^2)\setminus\gamma$. By Seifert-van Kampen theorem, the knot group of $B(\omega,t,b)$ is generated by $x$ and $y$, with a relation $r_\mu r_\lambda=r_\lambda r_\mu$, where $r_\mu$ and $r_\lambda$ are elements generated by $x$ and $y$, which represent a homotopical meridian and a homotopical longitude of $S^1\times D^2$ which does not intersect the arc $\gamma$, i.e., we have
$$\pi_1(N(S^1 \times D^2)\setminus (\gamma \cup \delta))= \langle x,y | r_\mu r_\lambda= r_\lambda r_\mu,r_\mu=r_\mu(x,y),r_\lambda=r_\lambda(x,y) \rangle .$$ 

By definition, the element 
$$\mu= x y^{-1}$$
represent a meridian of $B(\omega,t,b)$.

To find the functions $r_\mu(x,y)$ and $r_\lambda(x,y)$, we first consider a meridian and a longitude of $S^1\times D^2$. Assume that the meridian intersects the arc $\gamma$ in $\omega$ points, namely $Q_1, Q_2,\ldots, Q_\omega$, and it does not intersect $\partial D$. Assume that the longitude intersects $\partial D_y$ positively once and then intersects the arc $\gamma$ in $t$ points, namely $P_t, P_{t-1},\ldots, P_1$ in order. Then, we isotope the curve so that the intersection points $Q_1,Q_2,\ldots,Q_{\omega}$ and $P_t, P_{t-1},\ldots,  P_1$ move in the positive direction of the arc $\gamma$ past its endpoint. The resulting curves are represented by $r_\mu$ and $r_\lambda$.

Let $P_0$ be the starting point of the arc $\gamma$. For each integer $i$ with $0\le i\le t$ (resp., each integer $j$ with $1\le j\le\omega$), in the process when the crossing point $P_i$ (resp., $Q_j$) moves in the positive direction of the arc $\gamma$ towards its endpoint, starting from identity, each time the point goes past $\partial D_x$ we right multiply the element by $x$, and each time the point goes past $\partial D_y$ we right multiply the element by $y$. Let the resulting element be $g_i$ (resp., $h_j$). Then we have
$$r_\mu = (h_1\mu h_1^{-1})(h_2\mu h_2^{-1})\cdots(h_\omega\mu h_\omega^{-1}), $$
and 
$$r_\lambda= y (g_t \mu^{-1} g_t^{-1})(g_{t-1} \mu^{-1}  g_{t-1}^{-1})\cdots (g_1  \mu^{-1}  g_1^{-1}).$$

Furthermore, we have the following properties. 
\begin{enumerate}
    \item $g_i$ ($1\le i\le t$) and $h_j$ ($1\le j\le \omega$) are suffixes of $g_0$; 
    \item If $i\equiv j\pmod{\omega}$, then $g_i = h_j$.
\end{enumerate}

The longitude of $B(\omega,t,b)$ is represented by $\lambda= y g_0 \mu^k$ for some integer $k$ to cancel out the linking number. In $H_1(N(S^1 \times D^2)\setminus (\gamma \cup \delta))/[r_\lambda]$, we have $0=[r_\lambda]= [y] - t[\mu]$, so $[y]=t[\mu]$. We also have $[\lambda]= [y]+[g_0]+k [\mu]=b[x]+(\omega-b)[y]+k[\mu]=\omega [y]+(k+b)[\mu] =(\omega t + k + b)[\mu]$. So we have $k=-\omega t-b$. Therefore we have
$$\lambda=y g_0 \mu^{-\omega t-b}.$$ 

\section{Partial order $\le_C$, root-closed, conjugacy-closed submonoid and preorder $\le_k$}

In this section, we will define a few terms and prove relevant theorems to describe our method of proving non-left-orderability. 

Let $G$ be a group with a left order $\le$. We define a binary relation $\le_C$ as follows. 

\begin{definition}
For $f,g \in G$, we say $f\le_C g$ when $f h\le g h$ for every $h\in G$.
\end{definition}

Then $\le_C$ has the following properties.

\begin{theorem}\label{partial-order-thm}
The binary relation $\le_C$ is a partial order which is invariant under left multiplication and right multiplication.
\end{theorem}
\begin{proof}
The reflexivity, transitivity and antisymmetry of $\le$ implies the reflexivity, transitivity and antisymmetry of $\le_C$ respectively. Therefore $\le_C$ is a partial order.

Suppose $f\le_C g$. By definition, $fh \le gh$ for every $h\in G$. Because $\le$ is invariant under left multiplication, we have $h' f h\le h' g h$ for every $h', h \in G$. Because $h$ is arbitrary, we replace $h$ with $h'' h$ to deduce that $h' f h'' h\le h' f h'' h$ for every $h',h'',h\in G$. By definition, $h'f h''\le_C h'f h''$ for every $h',h''\in G$.  
\end{proof}

\begin{theorem}\label{partial-order-root}
For $n\in \mathbf{N}^+$ and $f\in G$, if $f^n\le_C 1$, then $f\le_C 1$.
\end{theorem}
\begin{proof}
Suppose $f^n\le_C 1$. Let $g$ be an arbitrary element in $G$. By definition, we have $f^n g \le g$. 

If $g\le fg$, then $f^k g\le f^{k+1}g$ for any integer $k$. Hence we have $fg \le f^2 g\le \cdots \le f^n g\le g$. Since we have $fg\le g$ if $g\le fg$, and $\le$ is a total order, we always have $fg\le g$. By definition, we have $f\le_C 1$.
\end{proof}

\begin{theorem}\label{partial-order-extension}
The total order $\le$ is a linear extension of $\le_C$.
\end{theorem}

\begin{proof}
By taking $h=1$ in the definition of $\le_C$, we have that $f\le_C g$ implies $f\le g$. 
\end{proof}

We summarize the properties of its positive cone in the following definition.

\begin{definition}
The root-closed, conjugacy-closed submonoid $M$ of $G$ generated by a set $S$ is defined as the minimal subset of $G$ containing $S$ and identity, with the following properties:
\begin{enumerate}[label=(\alph*)]
    \item If $f,g\in M$, then $fg \in M$;
    \item If $f\in M$ and $g\in G$, then $g f g^{-1}\in M$;
    \item If $n\in \mathbf{N}^+$ and $f^n\in M$, then $f\in M$.
\end{enumerate}
\end{definition}

It was shown \cite{Hol} that there exists a monomorphism $\rho: G\to \mbox{Homeo}_+(\mathbf{R})$ without global fixed points, such that $f\le g$ if and only of $\rho(f)(0)\le \rho(g)(0)$ for every $f,g \in G$. Suppose that, for some specific set $S_k$ of elements $g$ in $G$, we have proved the homeomorphisms $\rho(g)$ $(g\in S_k)$ are either all pointwise not less than the identity map, or all pointwise not greater than the identity map. Then we define another useful binary relation as follows.

\begin{definition}
For $f,g \in G$, we say $f\le_k g$ when $f^{-1} g$ is in the root-closed, conjugacy-closed submonoid generated by the subset $S_k$ of $G$.
\end{definition}

Notice that $\le_k$ does not depend on the choice of the left order $\le$, but only depend on the generating set $S_k$. The following two properties of $\le_k$ are similar to Theorem \ref{partial-order-thm} and Theorem \ref{partial-order-root}, and they even do not rely on the existence of a left order.

\begin{theorem}\label{preorder-thm}
The binary relation $\le_k$ is a preorder which is invariant under left multiplication and right multiplication.
\end{theorem}
\begin{proof}
Let the root-closed, conjugacy-closed submonoid generated by $S_k$ be $M_k$. $M_k$ has an identity element, so $\le_k$ is reflexive. $M_k$ is closed under multiplication, so $\le_k$ is transitive. Therefore $\le_k$ is a preorder.

Suppose $f\le_k g$. When $f$ and $g$ are both left multiplied by an element, $f^{-1} g$ does not change. So $\le_k$ is invariant under left multiplication. When $f$ and $g$ are both right multiplied by an element, $f^{-1} g$ changes to its conjugation. $M_k$ is closed under conjugation, so $\le_k$ is invariant under right multiplication.
\end{proof}

\begin{theorem}\label{preorder-root}
For $n\in \mathbf{N}^+$ and $f\in G$, if $f^n\le_k 1$, then $f\le_k 1$.
\end{theorem}
\begin{proof}
Since a root-closed, conjugacy-closed submonoid is closed under taking roots, the theorem holds.
\end{proof}

For a left orderable group $G$, the preorder $\le_k$ has the following property, which is similar to Theorem \ref{partial-order-extension}.

\begin{theorem}\label{extension-thm}
For the group $G$ with a left order $\le$ and a compatible monomorphism $\rho$, suppose that the homeomorphisms $\rho(g)$ $(g\in S_k)$ are either all pointwise not less than the identity map, or all pointwise not greater than the identity map. Then the partial order $\le_C$ is either an extension of $\le_k$, or an extension of the inverse of $\le_k$. 
\end{theorem}
\begin{proof}
If the homeomorphisms $\rho(g)$ $(g\in S_k)$ are all pointwise not less than the identity map, then for an arbitrary choice of $h\in G$, we have $\rho(gh)(0)= \rho(g)\rho(h)(0)>\rho(g)(0)$. By definition of $\rho$, we have $gh>h$ for every $g\in S_k$ and $h\in G$, that is, $g >_C 1$ for any $g\in S_k$. By Theorem \ref{partial-order-thm} and Theorem \ref{partial-order-root}, $h\ge_C 1$ for every $h$ in the root-closed, conjugacy-closed submonoid generated by $S_k$. By taking $h=f^{-1}g$, we have that $f\le_k g$ implies $f\le_C g$.

If the homeomorphisms $\rho(g)$ $(g\in S_k)$ are all pointwise not greater than the identity map, similarly we have that $f\le_k g$ implies $g\le_C f$.
\end{proof}

\begin{corollary}\label{preorder-partial-order}
With the conditions of Theorem \ref{extension-thm}, the preorder $\le_k$ is a partial order.
\end{corollary}
\begin{proof}
The antisymmetry of $\le$ implies the antisymmetry of $\le_k$.
\end{proof}

\section{Nontrivial knots which are closures of $1$-bridge braids have property (D)}\label{main-section}
In this section, we will define the property (D) and prove relevant theorems. We will also complete the proof of Theorem \ref{main} and Theorem \ref{main-variant}.

As in Section \ref{introduction}, we define property (D) for nontrivial knots as follows.

\begin{definition}
For a nontrivial knot $K$ with $\mu$ and $\lambda$ representing a meridian and a longitude in the knot group, we say $K$ has property (D) if 
\begin{enumerate}[label=(\alph*)]
    \item for any homomorphism $\rho$ from $\pi_1(S^3\setminus K)$ to $\mbox{Homeo}_+(\mathbf{R})$, if $s\in\mathbf{R}$ is a common fixed point of $\rho(\mu)$ and $\rho(\lambda)$, then $s$ is a fixed point of every element in $\pi_1(S^3\setminus K)$, and,
    \item $\mu$ is in the root-closed, conjugacy-closed submonoid generated by $\mu^{2g(K)-1}\lambda$ and $\mu^{-1}$.
\end{enumerate}
\end{definition}

Now we prove that property (D) implies the non-left-orderability.

\begin{theorem}\label{D-implies-not-orderable}
Suppose a nontrivial knot $K$ has property (D), then the fundamental group of the manifold obtained by Dehn surgery on $K$ with slope $\frac{p}{q}\ge 2g(K)-1$ is not left orderable.
\end{theorem}
\begin{proof}
Let $G$ denote the fundamental group of the manifold obtained by Dehn surgery on $K$ with slope $\frac{p}{q}\ge 2(K)-1$. Let $\mu$ and $\lambda$ represent a meridian and a longitude. Let $\rho: G\to \mbox{Homeo}_+(\mathbf{R})$ be a monomorphism without global fixed points. As a pullback by inclusion, $\rho$ induces a homomorphism from $\pi_1(S^3\setminus K)$ to $\mbox{Homeo}_+(\mathbf{R})$ without global fixed points. Since $\mu^p=\lambda^{-q}$, we have $\rho(\mu)^p=\lambda(\lambda)^{-q}$, so $\rho(\mu)$ and $\rho(\lambda)$ have the same set of fixed points. Because $K$ has property (D), $\rho(\mu)$ does not have fixed points.

By continuity, either $\rho(\mu)$ is pointwise greater than identity map or $\rho(\mu)$ is pointwise less than identity map. We define the preorder $\le_k$ on $G$ generated by $\mu$. Since $1=\mu^p\lambda^q=\mu^{p-(2g(K)-1)q}(\mu^{2g(K)-1} \lambda)^q\ge_k (\mu^{2g(K)-1} \lambda)^q$, by Theorem \ref{preorder-root}, we have $\mu^{2g(K)-1} \lambda \le_k 1$. Since $\mu^{-1}\le_k 1$ ,$\mu^{2g(K)-1} \lambda \le_k 1$ and $\mu$ is in the root-closed, conjugacy-closed submonoid generated by $\mu^{2g(K)-1}\lambda$ and $\mu^{-1}$, we have $\mu\le_k 1$. As $\rho(\mu)$ is different from the identity map, we know $\mu$ and $1$ are different, but $\mu\le_k 1$ and $1\le_k \mu$. Therefore $\le_k$ is not a partial order. By Corollary \ref{preorder-partial-order}, $G$ is not left orderable.
\end{proof}

Next, we will prove that closures of $1$-bridge braids have property (D) (Theorem \ref{main-variant}). Let the nontrivial knot $K$ be the closure of a $1$-bridge braid $B(\omega,t,b)$. The fundamental group $G$, meridian $\mu$, longitude $\lambda$ are defined in the proof of Theorem \ref{D-implies-not-orderable}. As in Section \ref{knot_group_sectiion}, the group $G$ has a presentation
$$G= \langle x,y| r_\lambda=1, \mu^p \lambda^q =1\rangle.$$

Before we prove the part (a) of the property (D), we establish a theorem which serves as a step towards a global fixed point.  

\begin{theorem}\label{fixed-point-general-rule}
Suppose $\rho$ is a homomorphism from a group $G$ to $\mbox{Homeo}_+(\mathbf{R})$. Let $g_1,g_2$ be elements in $G$, $g_3$ be represented by a word with alphabet $\{g_1,g_2\}$ and at least one letter $g_1$, and $g_4$ be represented by a word with alphabet $\{g_1^{-1},g_2\}$ and at least one letter $g_1^{-1}$. If $s\in \mathbf{R}$ is a common fixed point of $\rho(g_3)$ and $\rho(g_4)$, then $s$ is a fixed point of $\rho(g_1)$.
\end{theorem}
\begin{proof}
The element $g_3$ is represented by a word with alphabet $\{g_1,g_2\}$ and at least one letter $g_1$, hence the first condition is that $g_1$ is in the monoid generated by $\{g_1^{-1},g_2^{-1},g_3\}$. The element $g_4$ is represented by a word with alphabet $\{g_1^{-1},g_2\}$ and at least one letter $g_1^{-1}$, hence the second condition is that $g_1$ is in the monoid generated by $\{g_1^{-1},g_2,g_4^{-1}\}$. 

Suppose that $s$ is a real number with $\rho(g_3)s=\rho(g_4)s=s$.

If $\rho(g_2) s\ge s$ and $\rho(g_1)s\ge s$, then by the first condition we have $\rho(g_1)s \le s$.

If $\rho(g_2) s\ge s$ and $\rho(g_1)s\le s$, then by the second condition we have $\rho(g_1)s \ge s$.

If $\rho(g_2) s\le s$ and $\rho(g_1)s\ge s$, then by the second condition we have $\rho(g_1)s \le s$.

If $\rho(g_2) s\le s$ and $\rho(g_1)s\le s$, then by the first condition we have $\rho(g_1)s \ge s$.

In all cases, we have $\rho(g_1)s=s$.
\end{proof}

Now, we prove the part (a) of the property (D).

\begin{lemma}\label{no-fix-point-main}
For any homomorphism $\rho$ from $\pi_1(S^3\setminus K)$ to $\mbox{Homeo}_+(\mathbf{R})$, if $s\in \mathbf{R}$ is a common fixed point of $\rho(\mu)$ and $\rho(\lambda)$, then $s$ is a fixed point of every element in $\pi_1(S^3\setminus K)$.
\end{lemma}

\begin{proof}
Since $s$ is a fixed point of $\rho(xy^{-1})=\rho(\mu)$ and $\rho (y g_0)=\rho(\mu^{\omega t+b} \lambda)=\rho(\mu)^{\omega t+b} \rho(\lambda)$, and $y g_0$ is represented by a word with $b$ letter $x$'s and $\omega-b$ letter $y$'s, by Theorem \ref{fixed-point-general-rule}, we have $\rho(y)s=s$. Since $x=\mu y$, we have $\rho(x)s=s$. Since $G$ is generated by $x$ and $y$, $s$ is a fixed point of every element in $\pi_1(S^3\setminus K)$.
\end{proof}

Then, we prove the part (b) of the property (D).

Let $\tilde{g}_0=1,\tilde{g}_1,\ldots, \tilde{g}_{\omega-1}=g_0, \tilde{g}_\omega= y g_0$ be all suffixes of the word $y g_0$, ordered by length. For each integer $i$ with $0\le i \le \omega-1$, suppose that $\tilde{g}_i$ appears $b_i$ times in $g_1, g_2,\ldots, g_t$. Let $l$ be the index with $\tilde{g}_l=g_1$.

We define the preorder $\le_k$ generated by $\mu$ and $(\mu^{2g(K)-1} \lambda)^{-1}$ on the knot group $$\pi_1(S^3\setminus K)= \langle x,y| r_\lambda=1\rangle.$$ Since $xy^{-1}=\mu$, we have $x\ge_k y$. For each $0\le i \le \omega-1$, since $\tilde{g}_{i+1}=x\tilde{g}_{i}$ or $\tilde{g}_{i+1}=y\tilde{g}_{i}$, we have $\tilde{g}_{i+1}\ge_k y \tilde{g}_i$. Because $r_\lambda=1$, we have 
\begin{align*}
y&=(g_1\mu g_1^{-1})(g_2 \mu g_2^{-1})\cdots (g_t\mu g_t^{-1}) \\
&\ge_k (g_1\mu g_1^{-1})(\tilde{g}_i \mu \tilde{g}_i^{-1})^{b_i}
\end{align*} for each integer $i$ with $0\le i\le \omega-1$ and $i\neq l$. Therefore \begin{align}\label{partial-inequality}
\begin{split}
    \tilde{g}_{i+1}&\ge_k y \tilde{g}_i \\&\ge_k g_1 \mu  g_1^{-1} \tilde{g}_i \mu^{b_i}. 
    \end{split}
\end{align}

By applying the inequality (\ref{partial-inequality}) for $0\le i \le l-1$, we have 
\begin{align*}
    g_1 &= \tilde{g}_{l}\\
    &\ge_k (g_1 \mu  g_1^{-1})^{l} \tilde{g}_{0} \mu^{\sum_{i=0}^{l-1} b_i}\\
    &= g_1 \mu^l  g_1^{-1} \mu^{\sum_{i=0}^{l-1} b_i},
\end{align*}
$$g_1\ge_k \mu^{l+\sum_{i=0}^{l-1} b_i}.$$

By applying the inequality (\ref{partial-inequality}) for $l+1\le i \le \omega-1$, we have 
\begin{align*}
    y g_0 &= \tilde{g}_{\omega+1}\\
    &\ge_k (g_1 \mu  g_1^{-1})^{\omega-l-1} \tilde{g}_{l+1} \mu^{\sum_{i=l+1}^{\omega-1} b_i}\\
    &= g_1 \mu^{\omega-l-1}  g_1^{-1} \tilde{g}_{l+1} \mu^{\sum_{i=l+1}^{\omega-1} b_i}\\
    &\ge_k \mu^{\omega-1+\sum_{i=0}^{l-1} b_i} g_1^{-1} y \tilde{g}_{l} \mu^{\sum_{i=l+1}^{\omega-1} b_i} \\
    &= \mu^{\omega-1+\sum_{i=0}^{l-1} b_i} g_1^{-1} y g_1 \mu^{\sum_{i=l+1}^{\omega-1} b_i}.  
\end{align*}

Since $2g(K)-1= \omega t+b-t-\omega$, have 
\begin{align*}
    y g_0 &=\mu^{\omega t+b} \lambda \\
     &=\mu^{t+\omega}(\mu^{2g(K)-1}\lambda)\\
     &\le_k \mu ^{t+\omega}.
\end{align*}

Since $\sum_{i=0}^{\omega-1} b_i=t$, we have $$ y g_1 \le_k g_1 \mu ^{b_l+1}.$$

We first consider the case where $t>b_l$. In this case, let $l'$ be an index with $0\le l'\le \omega-1$ and $l'\neq l$ such that $b_{l'}\ge \frac{t-b_l}{\omega-1}>0$. There exists an integer $j$ with $0\le j \le b_l$, such that 
\begin{align*}
    g_1 \mu ^{b_l+1}&\ge_k y g_1\\
    &=(g_1\mu g_1^{-1})(g_2 \mu g_2^{-1})\cdots (g_t\mu g_t^{-1}) g_1 \\
&\ge_k (g_1\mu g_1^{-1})^j (\tilde{g}_{l'} \mu \tilde{g}_{l'}^{-1}) (g_1\mu g_1^{-1})^{b_l-j} g_1\\
&= g_1 \mu^j g_1^{-1} \tilde{g}_{l'} \mu \tilde{g}_{l'}^{-1} g_1 \mu^{b_l-j},
\end{align*}
which is $$g_1 \mu g_1^{-1}\ge_k \tilde{g}_{l'} \mu \tilde{g}_{l'}^{-1}.$$
Therefore \begin{align*}
    y \tilde{g}_{l'}&=  (g_1\mu g_1^{-1})(g_2 \mu g_2^{-1})\cdots (g_t\mu g_t^{-1})\tilde{g}_{l'}\\
    &\ge_k (\tilde{g}_{l'} \mu \tilde{g}_{l'}^{-1})^{b_l+b_{l'}}\tilde{g}_{l'}\\
    &= \tilde{g}_{l'} \mu ^{b_l+b_{l'}}.
\end{align*}
Since $\tilde{g}_{l'}$ is a suffix of $\tilde{g}_{l'} y g_0$, we have
\begin{align*}
    \tilde{g}_{l'}\mu^{t+\omega} 
    &\ge_k \tilde{g}_{l'} y g_0 \\
    &\ge_k y^{\omega} \tilde{g}_{l'}\\
    &\ge_k \tilde{g}_{l'} \mu^{\omega(b_l+b_{l'})},
\end{align*}
which is
\begin{equation}\label{last-1}
    \mu^{t+\omega- \omega(b_l+b_{l'})}\ge_k 1.
\end{equation} 
Furthermore, we have \begin{align*}
    t+\omega- \omega(b_l+b_{l'})&\le t+\omega-\omega b_l+(b_l-t-b_{l'})\\
    &\le -(\omega-1)(b_l-1)\\
    &\le 0.
\end{align*}

Then we consider the edge case where $t=b_l$, that is, $y= (g_1\mu g_1^{-1})^t =g_1\mu^t g_1^{-1}$. Because $g_1$ is a suffix of $g_1 y g_0$, we have 
\begin{align*}
    g_1 \mu^{t+\omega}&\ge_k g_1 y g_0\\
    &\ge_k  y^{\omega} g_1\\
    &= (g_1\mu^t g_1^{-1})^\omega g_1\\
    &=g_1 \mu^{\omega t},
\end{align*}
which is \begin{equation}\label{last-2}
    \mu^{t+\omega-\omega t}\ge_k 1.
\end{equation}
The exponent is not positive if and only if $t>1$. Now assume $t=1$. Since $K$ is not trivial, we have $b\ge 1$ Because $g_1$ is a suffix of $g_1 y g_0$, there exists an integer $j$ with $0\le j\le \omega-1$, such that
\begin{align*}
    g_1 \mu^{\omega+1}&\ge_k g_1 y g_0\\
    &\ge_k y^j x y^{\omega-j-1} g_1\\
    &= (g_1\mu g_1^{-1})^j \mu (g_1\mu g_1^{-1})^{\omega-j} g_1\\
    &=g_1 \mu^j g_1^{-1} \mu g_1 \mu^{\omega-j},
\end{align*}
which is 
$$y\ge_k \mu.$$
Since $x= \mu y\ge_k \mu^2$, we have 
\begin{equation}\label{last-3}
    \mu^{\omega+1}\ge_k y g_0 \ge_k \mu^{\omega+b}\ge_k\mu^{\omega+1}.
\end{equation}

In every case, we have reached a loop (\ref{last-1}) (\ref{last-2}) (\ref{last-3}) in the preorder $\le_k$. By the definition of $\le_k$, we can construct a computation tree where each leaf is $1$, $\mu$ or $(\mu^{2g(K)-1} \lambda)^{-1}$, each non-leaf node represents multiplication, conjugation or taking roots, and the root node is $1$. Furthermore, although there are many edge cases, we leave it to readers to check that in all cases, $\mu$ appears in the leaves at least once. Then every node in the computation tree is in the root-closed, conjugacy-closed submonoid generated by $\mu$ and $(\mu^{2g(K)-1} \lambda)^{-1}$. We prove the following statement.

\begin{lemma}
Every ancestor of a leaf $\mu$ is in the root-closed, conjugacy-closed submonoid generated by $\mu^{-1}$ and $\mu^{2g(K)-1} \lambda$.
\end{lemma}
\begin{proof}
Use induction on the depth of the node. When the depth is $0$, the node is the root node, which is identity. Suppose that the lemma is proved for nodes of depth $d-1$, consider a node $P$ of depth $d$. Then consider the parent of $P$, we have three cases.
First, suppose it is a multiplication node. Then $P$ is the multiplication of its parent and the inverse of its sibling. Since its parent is in the root-closed, conjugacy-closed submonoid generated by $\mu^{-1}$ and $\mu^{2g(K)-1} \lambda$, and its sibling is in the root-closed, conjugacy-closed submonoid generated by $\mu$ and $(\mu^{2g(K)-1} \lambda)^{-1}$, the statement is true for $P$.
Second, suppose it is a conjugation node. Then $P$ is a conjugate element of its parent. Therefore the statement is true for $P$.
Last, suppose it is a taking root node. Then $P$ is a positive power of its parent. Therefore the statement is true for $P$.
\end{proof}
\begin{corollary}\label{1-bridge-braid-have-D-part-b}
$\mu$ is in the root-closed, conjugacy-closed submonoid generated by $\mu^{-1}$ and $\mu^{2g(K)-1} \lambda$.
\end{corollary}

Combining Lemma \ref{no-fix-point-main} and Corollary \ref{1-bridge-braid-have-D-part-b}, we proved Theorem \ref{main-variant}. 

Combining Theorem \ref{main-variant} and Theorem \ref{D-implies-not-orderable}, we proved Theorem \ref{main}.

\section{Satellite knots with $1$-bridge braid patterns} \label{satellite-section}

In this section, we will prove Theorem \ref{main-satellite}.

Let $P$ denote the $1$-bridge braid pattern $B(\omega,t,b)$, and $K$ denote a nontrivial knot with property (P). Suppose that $2g(K)-1\le \frac{t}{\omega}$. Let $\mu_K$ and $\lambda_K$ represent a meridian and a longitude of $K$ in the knot group of $K$. Let $\mu$ and $\lambda$ represent a meridian and a longitude of $P(K)$ in the knot group of $P(K)$. As in Section \ref{knot_group_sectiion}, by Seifert–van Kampen theorem, the knot group of $P(K)$ is
$$ \pi_1 (S^3\setminus P(K))= (\pi_1(S^3\setminus K)\ast \langle x,y \rangle)/(\mu_K=r_\mu, \lambda_K=r_\lambda). $$

First, we prove the part (a) of the property (D).

\begin{lemma}\label{no-fix-point-satellite}
For any homomorphism $\rho$ from $\pi_1(S^3\setminus P(K))$ to $\mbox{Homeo}_+(\mathbf{R})$, if $s\in \mathbf{R}$ is a common fixed point of $\rho(\mu)$ and $\rho(\lambda)$, then $s$ is a fixed point of every element in $\pi_1(S^3\setminus P(K))$.
\end{lemma}

\begin{proof}
By reasoning similar to the proof of Lemma \ref{no-fix-point-main}, we have $\rho(x)s=\rho(y)s=s$. Since $\mu_K=r_\mu$ and $\lambda_K=r_\lambda$ are in the subgroup generated by $x$ and $y$, we have $\rho(\mu_K)s =\rho(\lambda_K)s= s$. Since $K$ has property (D), $s$ is a fixed point of every element in $\pi_1(S^3\setminus K)$. Because $\pi_1(S^3\setminus P(K))$ is generated by $x$, $y$ and elements of $\pi_1(S^3\setminus K)$, $s$ is a fixed point of every element in $\pi_1(S^3\setminus P(K))$.
\end{proof}

Then, we prove the part (b) of the property (D).

We define the preorder $\le_k$ generated by $\mu$ and $(\mu^{2g(P(K))-1} \lambda)^{-1}$ on the knot group $\pi_1(S^3\setminus P(K))$. Since $xy^{-1}=\mu$, we have $x\ge_k y$. For each $0\le i \le \omega-1$, since $\tilde{g}_{i+1}=x\tilde{g}_{i}$ or $\tilde{g}_{i+1}=y\tilde{g}_{i}$, we have $\tilde{g}_{i+1}\ge_k y \tilde{g}_i$.

As in Section \ref{knot_group_sectiion}, we have
$$r_\mu = (h_1\mu h_1^{-1})(h_2\mu h_2^{-1})\cdots(h_\omega\mu h_\omega^{-1}),$$
$$r_\lambda= y (g_t \mu^{-1} g_t^{-1})(g_{t-1} \mu^{-1}  g_{t-1}^{-1})\cdots (g_1  \mu^{-1}  g_1^{-1}).$$
We also have that $g_i = h_j$ if $i\equiv j\pmod{\omega}$. Therefore, we have
$$r_\mu^{2g(K)-1} r_\lambda = y (g_t\mu^{-1 }g_t^{-1})(g_{t-1}\mu^{-1 }g_{t-1}^{-1})\cdots(g_{\omega (2g(K)-1)+1}\mu^{-1 }g_{\omega (2g(K)-1)+1}^{-1})$$
for each integer $i$ with $\omega(2g(K)-1)<i\le t$. Here we use the condition $2g(K)-1\le \frac{t}{\omega}$.

As in Section \ref{main-section}, let $\tilde{g}_0=1,\tilde{g}_1,\ldots, \tilde{g}_{\omega-1}=g_0, \tilde{g}_\omega= y g_0$ be all suffixes of the word $y g_0$, ordered by length. For each integer $i$ with $0\le i \le \omega-1$, suppose that $\tilde{g}_i$ appears $b_i$ times in $g_{\omega (2g(K)-1)+1}, g_{\omega (2g(K)-1)+2},\ldots, g_t$. Then we have
\begin{align*}
    r_\mu^{2g(K)-1} r_\lambda&\le_k y (\tilde{g}_i \mu^{-1}\tilde{g}_i^{-1})^{b_i}\\
    &=y \tilde{g}_i\mu^{-b_i}\tilde{g}_i^{-1},
\end{align*}
which implies
\begin{align}\label{partial-inequality-satellite}
\begin{split}
    \tilde{g}_{i+1} &\ge_k y\tilde{g}_{i}\\
    &\ge_k (r_\mu^{2g(K)-1} r_\lambda)\tilde{g}_i \mu^{b_i}
\end{split}
\end{align}
for each integer $i$ with $0\le i \le \omega-1$.

By applying the inequality (\ref{partial-inequality-satellite}) for $0\le i \le \omega-1$, we have
\begin{align*}
    \mu^{\omega t + b}\lambda&= yg_0\\
    &=\tilde{g}_\omega\\
    &\ge_{k} (r_\mu^{2g(K)-1} r_\lambda)^\omega \tilde{g}_0 \mu^{\sum_{i=0}^{\omega-1} b_i}\\
    &=(r_\mu^{2g(K)-1} r_\lambda)^\omega \mu^{\sum_{i=0}^{\omega-1} b_i}
\end{align*}

By \cite[Theorem 5.8.1]{Cro2}, we have $$g(P(K))\ge g(P)+\omega g(K).$$
Because 
$$\sum_{i=0}^{\omega-1} b_i=t-\omega(2g(K)-1)$$
and
$$g(P)=\frac{1}{2}((\omega-1)(t-1)+b),$$
we have
\begin{align*}
\omega t+b-\sum_{i=0}^{\omega-1} b_i &=    \omega t-t-\omega+b+2\omega g(K)\\
&=2g(P)+2\omega g(K)-1\\
&\le 2g(P(K))-1.
\end{align*}
Therefore, we have 
\begin{align*}
    (r_\mu^{2g(K)-1} r_\lambda)^\omega&\le_k \mu^{\omega t+b-\sum_{t=0}^{\omega-1} b_i} \lambda\\
    &\le_k \mu^{2g(P(K))-1}\lambda\\
    &\le_k 1.
\end{align*}
By Theorem \ref{preorder-root}, we have $r_\mu^{2g(K)-1} r_\lambda\le_k 1$. 

Because $K$ has property (D), $r_\mu$ is in the root-closed, normal submonoid generated by $r_\mu^{2g(K)-1} r_\lambda$ and $r_{\mu}^{-1}$, as this relation is invariant under free multiplication and quotient maps. Therefore we have $r_\mu\le_k 1$, as $r_\mu^{2g(K)-1} r_\lambda\le_k 1$ and $r_{\mu}^{-1}\le_k 1$.

Since $r_\mu$ is the product of $\omega$ conjugates of $\mu$, $\mu$ is in the root-closed, normal submonoid generated by $r_\mu$ and $\mu^{-1}$. Therefore we have $\mu\le_k 1$, as $r_\mu\le_k 1$ and $\mu^{-1}\le_k 1$.

By the definition of $\le_k$, we completed the proof of the part (a) of the property (D). Combined with Lemma \ref{no-fix-point-satellite}, we proved Theorem \ref{main-satellite}.

\section{Non-left-orderability for Dehn fillings on the manifold $v2503$}\label{v2503-section}

According to the software SnapPy \cite{CDW} and Varvarezos's computation \cite{Var}, the fundamental group of the manifold $v2503$ has the presentation
$$\pi_1(v2503)=\langle a,b| a^2 b^{-2} a b^{-2} a ^2 b a^2 b a b a^2 b=1\rangle,$$
and a meridian $\mu$ and a longitude $\lambda$ can be written as
$$\mu = b^{-2} a b^{-2} a b^{-1},$$
and
$$\lambda = a^{-2} b^{-1} a^{-2} b.$$

The following lemma is similar to Lemma \ref{no-fix-point-main} and Lemma \ref{no-fix-point-satellite}.

\begin{lemma}\label{no-fix-point-v2503}
For any homomorphism $\rho$ from $\pi_1(v2503)$ to $\mbox{Homeo}_+(\mathbf{R})$, if $s\in\mathbf{R}$ is a fixed point of $\rho(\mu)$, then $s$ is a fixed point of every element in $\pi_1(v2503)$.
\end{lemma}
\begin{proof}
Since
\begin{align*}
    1&=a^2 b^{-2} a b^{-2} a ^2 b a^2 b a b a^2 b\\
    &=a^2\mu b a b a^2 b a b a^2 b,
\end{align*}
we have $$\mu^{-1}=b a b a^2 b a b a^2 b a^2.$$

By taking $g_1=b, g_2=a, g_3=\mu^{-1}, g_4=\mu$ in Theorem \ref{fixed-point-general-rule}, we get $\rho(b)s=s$.

By taking $g_1=a, g_2=b, g_3=\mu^{-1}, g_4=\mu^{-1}$ in Theorem \ref{fixed-point-general-rule}, we get $\rho(a)s=s$.

Since the group $\pi(v2503)$ is generated by $a$ and $b$, $s$ is a fixed point of every element in $\pi_1(v2503)$.
\end{proof}

The following theorem is an analog of the part (b) in the definition of property (D).

\begin{theorem}\label{has-fix-point-v2503}
$\lambda$ is in the root-closed, conjugacy-closed submonoid generated by $\lambda^{-1}$.
\end{theorem}

\begin{proof}
We define the preorder $\le_k$ generated by $\lambda$ on $\pi_1(v2503)$. Then $\lambda\ge_k 1$ implies that $a^2 b a^2\le_k b$. Therefore, we have
\begin{align*}
     1&=a^2 b^{-2} a b^{-2} a ^2 b a^2 b a b a^2 b\\
     &=a^2 b^{-2} a b^{-2} (a ^2 b a^2) b a b (a^2 b a^2)a^{-2}\\
     &\le_k a^2 b^{-2} a b^{-2} b b a b b a^{-2}\\
     &=a^2 b^{-2} a^2 b^2 a^{-2},
\end{align*}
which is $a^2\ge_k 1$. Therefore, we have
\begin{align*}
    \lambda &= a^{-2} b^{-1} a^{-2} b\\
    &\le_k b^{-1} b\\
    &=1.
\end{align*}
By the definition of preorder $\le_k$, $\lambda$ is in the root-closed, conjugacy-closed submonoid generated by $\lambda^{-1}$.
\end{proof}

Now we prove Theorem \ref{main-v2503}.

By Seifert-van Kampen theorem, the fundamental group of an manifold obtained by Dehn filling on $v2503$ with slope $\frac{p}{q}$, denoted by $G$, is the quotient group of $\pi_1(v2503)$ by the relation $\mu^p\lambda^q=1$. Let $\rho: G\to \mbox{Homeo}_+(\mathbf{R})$ be a monomorphism without global fixed points. As a pullback by inclusion, $\rho$ induces a homomorphism from $\pi_1(v2503)$ to $\mbox{Homeo}_+(\mathbf{R})$ without global fixed points. Since $\mu^p=\lambda^{-q}$, we have $\rho(\mu)^p=\rho(\lambda)^{-q}$. By Lemma \ref{no-fix-point-v2503}, $\rho(\mu)$ does not have fixed points.

If $p\neq 0$, then $\rho(\lambda)$ does not have fixed points. By Corollary \ref{preorder-partial-order}, $\le_k$ is a partial order. By Theorem \ref{has-fix-point-v2503}, we have $1\le_k \lambda\le_k 1$, so $\lambda=1$ which contradicts to that $\rho(\lambda)$ does not have fixed points.

If $p=0$, then we get the equalities $\lambda=1$, $a^2 b a^2=b$ and $a^2=1$. Then $G$ can be simplified to $\mathbf{Z}\ast(\mathbf{Z}/2\mathbf{Z})$. The existence of a nontrivial torsion element implies the non-left-orderability.

Therefore, Theorem \ref{main-v2503} holds.

\section{L-space twisted torus knots of form $T_{p,kp\pm 1}^{l,m}$ are closures of $1$-bridge braids}\label{braid-word-section}

In this section, we will complete the proof of Theorem \ref{main-braid-word}.

For any positive integer $m$, let $\pi_m$ denote the word $\sigma_m \sigma_{m-1}\cdots \sigma_1$, and let $\Pi_m$ denote the word $\pi_1 \pi_2\cdots \pi_m$. As $\Pi_m$ denotes the half twist on the first $m-1$ strands and $\pi_m^{m+1}$ denotes the full twist on the first $m-1$ strands, we have $$\pi_m^{m+1}=\Pi_m^2.$$ We also have $$\sigma_i \pi_m^{m+1}=\pi_m^{m+1}\sigma_i$$ and $$\sigma_i\Pi_m=\Pi_m\sigma_{m+1-i}$$ for each integer $i$ with $1\le i\le m$. 

For each integer $i$ with $1\le i \le m-1$, we have \begin{align*}
    \sigma_i \pi_m&= \sigma_i (\sigma_m \sigma_{m-1}\cdots \sigma_1)\\
    &= (\sigma_m \sigma_{m-1}\cdots\sigma_{i+2})(\sigma_i \sigma_{i+1} \sigma_{i})(\sigma_{i-1} \sigma_{i-2}\cdots \sigma_1)\\
    &= (\sigma_m \sigma_{m-1}\cdots\sigma_{i+2})(\sigma_{i+1} \sigma_{i} \sigma_{i+1})(\sigma_{i-1} \sigma_{i-2}\cdots \sigma_1)\\
    &=(\sigma_m \sigma_{m-1}\cdots \sigma_1)\sigma_{i+1}\\
    &=\pi_m \sigma_{i+1}.
\end{align*}

Then we have 
\begin{align*}
\pi_m^2 &= \pi_m(\sigma_m\sigma_{m-1}\cdots\sigma_1)\\
&=\sigma_{m-1}\pi_m(\sigma_{m-1}\sigma_{m-2}\cdots\sigma_1)\\
&=(\sigma_{m-1}\sigma_{m-2})\pi_m(\sigma_{m-2}\sigma_{m-3}\cdots\sigma_1)\\
&=\cdots\\
&=(\sigma_{m-1}\sigma_{m-2}\cdots\sigma_1)\pi_m\sigma_1\\
&=\pi_{m-1}\pi_m \sigma_1,
\end{align*}
and 
\begin{align*}
    \sigma_m \pi_m^2 &= \sigma_m \pi_{m-1}\pi_m \sigma_1\\
    &=\pi_m^2 \sigma_1.
\end{align*}

For a positive integer $s$ with $2\le s\le m$, we have
\begin{align*}
    \pi_m^s&=\pi_{m-1}\pi_m\sigma_1 \pi_m^{s-2}\\
    &= \pi_{m-1}\pi_m^2 \sigma_2 \pi_m^{s-3}\\
    &= \cdots\\
    &= \pi_{m-1}\pi_m^{s-1}\sigma_{s-1}.
\end{align*}

Then we have 
\begin{align*}
    \pi_m^s&=\pi_{m-1}\pi_m^{s-1}\sigma_{s-1}\\
    &=\pi_{m-1}^2\pi_m^{s-2}\sigma_{s-2}\sigma_{s-1}\\
    &= \cdots\\
    &=\pi_{m-1}^{s-1}\pi_m(\sigma_1\sigma_2\cdots\sigma_{s-1}).
\end{align*}

By Markov braid theorem, the closure of two braids $f$ and $g$ represent the same oriented link if and only if one can be transform into another by a sequence of Markov moves. In this case, we say $f\sim g$.

Suppose that $t\le b\le \omega-1$, then we have
\begin{align*}
    \pi_b \pi_{\omega-1}^t&= \pi_b \pi_{\omega-2}^{t-1}\pi_{\omega-1} (\sigma_1\sigma_2\cdots\sigma_{t-1}) \\
    &\sim \pi_b \pi_{\omega-2}^{t} (\sigma_1\sigma_2\cdots\sigma_{t-1}) \\
    &=\pi_b \pi_{\omega-3}^{t-1}\pi_{\omega-2} (\sigma_1\sigma_2\cdots\sigma_{t-1})^2 \\
    &\sim \pi_b \pi_{\omega-3}^{t} (\sigma_1\sigma_2\cdots\sigma_{t-1})^2\\
    &\sim \cdots\\
    &\sim \pi_b^{t+1} (\sigma_1\sigma_2\cdots\sigma_{t-1})^{\omega-b-1}\\
    &\sim \cdots\\
    &\sim \pi_t^{t+1} (\sigma_1 \sigma_2\cdots\sigma_{t})^{b-t}(\sigma_1 \sigma_2\cdots\sigma_{t-1})^{\omega-b-1}\\
    &\sim \Pi_t \pi_t^{t+1} (\sigma_1 \sigma_2\cdots\sigma_{t})^{b-t}(\sigma_1 \sigma_2\cdots\sigma_{t-1})^{\omega-b-1} \Pi_t^{-1}\\
    &=\pi_t^{b+1}(\sigma_{t}\sigma_{t-1}\cdots \sigma_2)^{\omega-b-1} \\
    &=\pi_t^{b}\pi_{t-1}^{\omega-b-1} \pi_t\\
    &\sim \pi_{t-1}^{\omega-b-1}\pi_t^{b+1}.
\end{align*}

When $q\ge l+1 =p$, we have $\pi_{l-1}^n \pi_{p-1}^q=\pi_{p-2}^n \pi_{p-1}^q \sim \pi_{q-1}\pi_{n+q-1}^{p-1}$. Therefore Theorem \ref{main-braid-word} holds with condition (a).

When $q=l$ divides $n$ and $\gcd(p,q)=1$, by definition, the twisted torus knot $T_{p,q}^{q,\frac{n}{q}}$ is the torus knot $T_{p+n,q}$. So Theorem \ref{main-braid-word} holds with condition (b).

Now we consider the condition (c). If $p=2$, then the twisted torus knot $T_{2,q}^{2,1}$ is the torus knot $T_{2,q+2}$. If $q=1$, then the twisted torus knot $T_{p,1}^{2,1}$ is the torus knot $T_{2,3}$. If $q=2$, then the twisted torus knot $T_{p,2}^{2,1}$ is the torus knot $T_{p+2,2}$. So we assume $p,q\ge 3$.

we prove the following elementary number theoretic lemma.

\begin{lemma}
Suppose $p,q\ge 3$ are coprime positive integers. Let $x$ be the multiplicative inverse of $p$ modulo $q$ with $0\le x<q$, and $y$ be the multiplicative inverse of $q$ modulo $p$ with $0\le y<p$. Then $x < \frac{q}{2}$ or $y < \frac{p}{2}$.
\end{lemma}
\begin{proof}
The integer $2pq-px -qy + 1$ is a multiple of $p$ and a multiple of $q$, so $2pq-px -qy + 1$ is a multiple of $pq$. If $\frac{q}{2}\le x\le q-1$ and $\frac{p}{2}\le y\le p-1$, as either $x\ge\frac{q+1}{2}$ or $y\ge \frac{p+1}{2}$, we have $0<p+q+1 \le 2pq-xq-yp+1 \le pq-\frac{1}{2}\min(p,q)+1 <pq$, which is a contradiction. 
\end{proof}

For each integer $i,j$ with $1\le i,j\le p-1$ and $j\equiv i+q\pmod{p}$, we have
\begin{align*}
    \sigma_i \pi_{p-1}^q&= \pi_{p-1}\sigma_{i+1} \pi_{p-1}^{q-1}\\
    &= \pi_{p-1}^2\sigma_{i+2} \pi_{p-1}^{q-2}\\
    &= \cdots\\
    &= \pi_{p-1}^{kp-1-i}\sigma_{p-1} \pi_{p-1}^{q-(kp-1-i)}\\
    &= \pi_{p-1}^{kp+1-i}\sigma_{1} \pi_{p-1}^{q-(kp+1-i)}\\
    &= \cdots \\
    &= \pi_{p-1}^q \sigma_{j}.
\end{align*}

If there exists a multiplicative inverse $y$ of $q$ modulo $p$ with $0\le y<\frac{p}{2}$, then $1+ s q\not \equiv 0 \pmod p$ for $s=1,2,\ldots,y-1$. Therefore there exists a sequence $1=i_0, i_1, i_2,\ldots, i_y=2$ such that $i_s\equiv i_{s-1}+q\pmod p$ for $s=1,2,\ldots,y$ and $3\le i_s\le p-1$ for $s=1,2,\ldots,y-1$.
\begin{align*}
    \sigma_1^2 \pi_{p-1}^q&=\sigma_1\sigma_{i_0} \pi_{p-1}^q\\
    &=\sigma_1 \pi_{p-1}^q \sigma_{i_1}\\
    &\sim  \sigma_{i_1} \sigma_1 \pi_{p-1}^q\\
    &= \sigma_1 \sigma_{i_1}  \pi_{p-1}^q\\
    &\sim\cdots \\
    &\sim  \sigma_{i_s} \sigma_1 \pi_{p-1}^q\\
    &=  \sigma_{2} \sigma_1 \pi_{p-1}^q\\
    &=\pi_2 \pi_{p-1}^q
\end{align*}

Suppose that a multiplicative inverse $x$ of $p$ modulo $q$ satisfies $0\le x<\frac{q}{2}$. By definition, the twisted torus knot $T_{p,q}^{2,1}$ is the twisted torus knot $T_{q,p}^{2,1}$. So we have
\begin{align*}
    \sigma_1^2 \pi_{p-1}^q &\sim \sigma_1^2\pi_{q-1}^p\\
    &\sim \pi_2 \pi_{q-1}^p.
\end{align*}

Therefore Theorem \ref{main-braid-word} holds with condition (c).

Since we have
\begin{align*}
    \pi_{p-3}^{p-2} \pi_{p-1}^{kp+1}
    &= \pi_{p-3}^{p-3} \pi_{p-1}^{kp+1}(\sigma_{p-2}\sigma_{p-3}\cdots \sigma_2)\\
    &\sim (\sigma_{p-2}\sigma_{p-3}\cdots \sigma_2)\pi_{p-3}^{p-3} \pi_{p-1}^{kp+1} \\    
    &= \pi_{p-2}\pi_{p-3}^{p-4}\pi_{p-4} \pi_{p-1}^{kp+1} \\
    &= \pi_{p-2}\pi_{p-3}^{p-4}\pi_{p-1}^{kp+1} 
    (\sigma_{p-3}\sigma_{p-4}\cdots \sigma_2)\\
    &\sim (\sigma_{p-3}\sigma_{p-4}\cdots \sigma_2)\pi_{p-2}\pi_{p-3}^{p-4}\pi_{p-1}^{kp+1} 
    \\
    &= \pi_{p-2}(\sigma_{p-2}\sigma_{p-3}\cdots \sigma_3)\pi_{p-3}^{p-4}\pi_{p-1}^{kp+1} 
    \\    
    &= \pi_{p-2}^2\pi_{p-3}^{p-5}\pi_{p-5}\pi_{p-1}^{kp+1} 
    \\    
    &\sim \cdots
    \\
    &\sim \pi_{p-2}^{p-3}\pi_{p-1}^{kp+1},
\end{align*}
Theorem \ref{main-braid-word} with condition (a) implies Theorem \ref{main-braid-word} with condition (d) and $q=kp+1$.

By definition, the twisted torus knot $T_{p,kp-1}^{p-3,1}$ is the twisted torus knot $T_{kp-1,p}^{p-3,1}$, so we have
\begin{align*}
    \pi_{p-3}^{p-2} \pi_{p-1}^{kp-1}
    &\sim \pi_{p-3}^{p-2} \pi_{kp-2}^{p}\\
    & = \pi_{p-3}^{p-3} \pi_{kp-2}^{p}(\sigma_{2p-3}\sigma_{2p-4}\cdots\sigma_{p+1})\\
    &\sim(\sigma_{2p-3}\sigma_{2p-4}\cdots\sigma_{p+1}) \pi_{p-3}^{p-3} \pi_{kp-2}^{p}\\
    &= \pi_{p-3}^{p-3} (\sigma_{2p-3}\sigma_{2p-4}\cdots\sigma_{p+1})\pi_{kp-2}^{p}\\
    &\sim\cdots\\
    &\sim \pi_{p-3}^{p-3} (\sigma_{kp-3}\sigma_{kp-4}\cdots\sigma_{(k-1)p+1})\pi_{kp-2}^{p}\\
    &= \pi_{p-3}^{p-3}\pi_{kp-2}^{p} (\sigma_{p-2}\sigma_{p-3}\cdots \sigma_2).
\end{align*}

As in the case with condition (d) and $q=kq+1$, we have
\begin{align*}    
\pi_{p-3}^{p-2} \pi_{p-1}^{kp-1}
    &\sim \pi_{p-3}^{p-2} \pi_{kp-2}^{p}\\
    &\sim \pi_{p-3}^{p-3}\pi_{kp-2}^{p} (\sigma_{p-2}\sigma_{p-3}\cdots \sigma_2)\\
    &\sim \cdots\\
    &\sim \pi_{p-2}^{p-3} \pi_{kp-2}^p.
\end{align*}

If $k=1$, then the closure of the braid is a torus knot. 

Suppose that $k\ge 2$. As in the case with condition (a), we have 
\begin{align*}
    \pi_{p-3}^{p-2} 
    &\sim \pi_{p-2}^{p-3} \pi_{kp-2}^p\\
    &\sim \pi_{p-2}^{p-3} \pi_{kp-3}^p (\sigma_1 \sigma_2\cdots \sigma_{p-1})\\
    &\sim \pi_{p-2}^{p-3} \pi_{kp-4}^p (\sigma_1 \sigma_2\cdots \sigma_{p-1})^2\\
    &\sim\cdots\\
    &\sim \pi_{p-2}^{p-3} \pi_{p-1}^p (\sigma_1 \sigma_2\cdots \sigma_{p-1})^{(k-1)p -1}\\
    &= \pi_{p-2}^{2p-4} \pi_{p-1} (\sigma_1 \sigma_2\cdots \sigma_{p-1})^{(k-1)p }\\
    &=\pi_{p-2}^{2p-4} \pi_{p-1} (\Pi_{p-1} \pi_{p-1}\Pi_{p-1}^{-1})^{(k-1)p }\\
    &=\pi_{p-2}^{2p-4} \pi_{p-1} \Pi_{p-1} \pi_{p-1}^{(k-1)p }\Pi_{p-1}^{-1}\\
    &=\pi_{p-2}^{2p-4} \pi_{p-1} \Pi_{p-1} \Pi_{p-1}^{2(k-1) }\Pi_{p-1}^{-1}\\
    &=\pi_{p-2}^{2p-4} \pi_{p-1} \Pi_{p-1}^{2(k-1) }\\
    &=\pi_{p-2}^{2p-4}\pi_{p-1}^{(k-1)p+1}.
\end{align*}
Theorem \ref{main-braid-word} with condition (a) implies Theorem \ref{main-braid-word} with condition (d) and $q=kp-1$.

Therefore, Theorem \ref{main-braid-word} is proved.

\end{document}